\documentclass[reqno]{amsart}
\usepackage{setspace}
\usepackage{amsthm}
\usepackage{amssymb,amsmath}
\usepackage{hyperref}
\newcommand{\CC}{\mathbb{C}}
\newcommand{\cH}{\mathcal{H}}
\newcommand{\F}{\mathcal{F}}
\newcommand{\FF}{\mathbb{F}}
\newcommand{\NN}{\mathbb{N}}
\newcommand{\op}[1]{\operatorname{#1}}
\newcommand{\QQ}{\mathbb{Q}}
\newcommand{\RR}{\mathbb{R}}
\newcommand{\TT}{\mathbb{T}}
\newcommand{\ZZ}{\mathbb{Z}}
\newcommand{\ip}[2]{\langle #1 \vert #2 \rangle}
\providecommand{\abs}[1]{\lvert#1\rvert}
\DeclareMathOperator{\Aut}{Aut}
\DeclareMathOperator{\e}{e}
\DeclareMathOperator{\re}{Re}
\swapnumbers
\newtheorem{theorem}{Theorem}[section]
\newtheorem{lemma}[theorem]{Lemma}
\newtheorem{corollary}[theorem]{Corollary}

\theoremstyle{definition}

\newtheorem{topic}[theorem]{}
\theoremstyle{remark}
\newtheorem{remark}[theorem]{Remark}
\begin{document}
\title[A Poincare series on hyperbolic space]{A Poincare series on hyperbolic space}
\author{Tathagata Basak}
\address{Department of Mathematics\\Iowa State University\\Ames, IA 50010}
\email{tathagat@iastate.edu}
\urladdr{http://orion.math.iastate.edu/tathagat}
\date{July 2, 2017}
\begin{abstract} Let $L$ be the unique even self-dual lattice of signature $(25,1)$.
The automorphism group $\operatorname{Aut}(L)$ acts on the hyperbolic space $\mathcal{H}^{25}$.
We study a Poincare series $E(z,s)$ defined for $z$ in $\mathcal{H}^{25}$,
convergent for $\operatorname{Re}(s) > 25$,
invariant under $\operatorname{Aut}(L)$ and having singularities along
the mirrors of the reflection group of $L$. We compute the Fourier
expansion of $E(z,s)$ at a ``Leech cusp" and prove that it can be meromorphically continued to
$\operatorname{Re}(s) > 25/2$.
Analytic continuation of Kloosterman sum zeta functions imply that
the individual Fourier coefficients of $E(z,s)$
have meromorphic continuation to the whole $s$-plane.
\end{abstract}
\maketitle
%
%
%
%
%
%
%
\section{Introduction}
\label{section-introduction}
Given positive integers $m$ and $n$ with $m  - n\equiv 0 \bmod 8$, let
$\mathrm{II}_{m,n}$ denote the unique even self-dual lattice of signature
$(m,n)$; see \cite{CS:SPLAG} chapter 26, 27.
For basic definitions about lattices, see
\cite{REB:L} or \cite{CS:SPLAG}. All the lattices considered here are
nonsingular unless otherwise stated. A bilinear form is usually denoted by $\ip{\;}{\;}$.
The lattice $\mathrm{II}_{1,1}$ is sometimes called a {\it hyperbolic cell}.
The even self-dual Lorentzian lattices are $\mathrm{II}_{8n + 1, 1}$. Among these,
the lattice $L = \mathrm{II}_{25,1}$ stands out
as exceptional  (see \cite{CS:SPLAG} or \cite{REB:Thesis}, \cite{REB:L}, 
for a wealth of information on $L$).
Of course this is related to the many exceptional properties of the
Leech lattice $\Lambda$. 
Indeed $ L \simeq \Lambda \oplus \mathrm{II}_{1,1}$.
Let $L(2)$ denote the set of norm $2$ vectors (also called {\it roots}) of $L$.
The roots form a single orbit under the automorphism group $\op{Aut}(L)$.
In this article, we study the Poincare-Weierstrass series
\begin{equation*}
E_L(z,s) = E(z,s) = \sum_{r \in L(2) } \ip{r}{z}^{-s}.
\end{equation*}
Here $z$ is a negative norm vector in $L \otimes \RR$ and
 $s$ is a complex
number.
The function $E(z,s)$ is real analytic in $z$. 
Complex analytic versions of this function
have appeared in Looijenga's study of compactifications of
complex ball quotients (see the functions $F^{(l)}_{\mathcal{O}}$ in lemma 5.4 of \cite{L:CDBA1}). 
The functions $F^{(l)}_{\mathcal{O}}$ are analogous to the 
Eisenstein series $\sum_{(m,n) \neq (0,0)} (m \tau + n)^{-2l}$ whereas
 $E(z,s)$ is analogous to real analytic Eisenstein series.
\par
The infinite series for $E(z,s)$ converges for $\op{Re}(s) > 25$
(For a conceptual proof, see lemma 5.4 of \cite{L:CDBA1}.
For the sake of completeness, we have included an elementary
proof in appendix \ref{l-convergence}).
So, for fixed $s$ with $\op{Re}(s) > 25$, we obtain a real analytic function
$z \mapsto E(z,s)$ on hyperbolic space
$\mathcal{H}^{25}$.
By definition, $E(z,s)$ is an automorphic function invariant under $\op{Aut}(L)$ with singularities
along the mirrors of the reflection group of $L$. 
Our objective is to compute the Fourier series for $E(z,s)$ and prove that
$E(z,s)$ can be analytically continued to a meromorphic function on $\op{Re}(s) > 25/2$.
\par
Much of what we say about $E_{\mathrm{II}_{25,1}}(z,s)$
can probably be generalized for all the even self-dual 
Lorentzian lattices $\mathrm{II}_{ 8 n + 1, 1}$
(at least if we assume $n \geq 3$, which would guarantee the existence of
a positive definite even self-dual lattice of rank $8n$ with no roots). 
However, in section \ref{section-II251} and \ref{section-Fourier-expansion}, while
studying $E_L(z,s)$, 
we have decided to restrict to the example $L = \mathrm{II}_{25,1}$ for several reasons.
First, it allows us to keep the exposition relatively simple.
Second, in view of moonshine, this is the most important example. 
Third, because of special properties of the Leech lattice, in particular, 
because of Conway's beautiful description of the automorphism group of $L$, 
this is also the nicest example. 
\par
To describe the form of the Fourier expansion of $E(z,s)$, we need to first
briefly recall Conway's description of the {\it reflection group} $R(L)$. By definition,
$R(L)$ is the subgroup of $\op{Aut}(L)$ generated by reflections in the roots of $L$.
Both $\op{Aut}(L)$ and $R(L)$ act on
the hyperbolic space
 $ B(L) = \lbrace x \in L \otimes \RR \colon x^2 < 0 \rbrace/ \RR^*
 \simeq \mathcal{H}^{25}$.
To describe a fundamental domain of $R(L)$ acting on $\mathcal{H}^{25}$,
choose a {\it Leech cusp} $\rho$:  this means that $\rho$ is a primitive 
norm zero vector of $L$ and $\rho^{\bot}/\rho \simeq \Lambda$.
This lets us split a hyperbolic cell from $L$ and 
identify $L = \Lambda \oplus \mathrm{II}_{1,1}$. So we write $v \in L$ in the form
$v = (\lambda; m,n)$ with $\lambda \in \Lambda$ and $m,n \in \ZZ$
with $v^2 = \lambda^2 - 2 m n$. In this co-ordinate system $\rho = (0; 0, 1)$.
The Leech cusp $\rho$ determines a point in $\partial \mathcal{H}^{25}$,
also denoted by $\rho$. The action of  $R(L)$ on $\mathcal{H}^{25}$
has a unique fundamental domain $C$ whose closure contains $\rho$.
We say that $C$ is the {\it Weyl chamber ``around" $\rho$}.
The walls of $C$ are in bijection with the vectors of Leech lattice. The angles
between these walls are determined by the inner products of the corresponding Leech lattice
vectors. This lets one describe $R(L)$ explicitly as a Coxeter group so
that
{\it ``the Leech lattice is the Dynkin diagram for $R(L)$"}.
(see \ref{t-C} or \cite{C:Aut26}, \cite{REB:Thesis} for details).
\par
Because the Leech lattice has no roots, there are no mirrors of $R(L)$
that pass through $\rho$. The subgroup $\TT$
of $\op{Aut}(L)$ that fix $\rho$ and act trivially on
$\rho^{\bot}/\rho$ is a free abelian group isomorphic to $\Lambda$.
We call $\TT$ the group of {\it translations}. 
We compute the Fourier series of $E(z,s)$ at $\rho$
with respect to the group of translations; see \ref{topic-continuous-group-of-translations},
\ref{topic-discrete-group-of-translations}, \ref{topic-setup-for-Fourier-series-computation}
 for details.
Fix positive real numbers $h,k$ such that $2h^2/k < 1$. Let $z \in L \otimes \RR$ 
be a vector of norm $-k$ having the form $z = (v h; h, *)$, where
$v \in \Lambda \otimes \RR$  and the last coordinate of $z$ is determined by the condition
$z^2 = -k$. Then $z$ (or rather its image in $\cH^{25}$)
is contained in the Weyl chamber $C$. So $z$ does not lie on any mirror.
The Fourier series of $E(z,s)$ has the form
\begin{equation*}
E((vh; h, *),s) = \sum_{\lambda \in \Lambda} a_{\lambda}(k,h,s)\exp( 2 \pi i \ip{\lambda}{v}).
\end{equation*}
Each coefficient in turn is an infinite series of the form
\begin{equation*}
a_{\lambda}(k,h,s) = \sum_{n = 1}^{\infty} j_{\lambda, n} n^{-s} g_{\lambda}(s,1/n)
\end{equation*}
where 
\begin{equation*}
j_{\lambda, n} = \sum_{l \in \Lambda/n \Lambda : l^2/2 \equiv 1 \bmod n} 
\exp( 2 \pi i \ip{l}{\lambda}/n) 
\end{equation*}
and $g_{\lambda}(s,y)$ is some function that is 
entire in $s$ and analytic for $y$ in a disc of radius bigger than $1$
and exponentially decaying as $\abs{\lambda}  \to \infty$; see theorem
\ref{th-fourier-coeffients}.
\par
Section \ref{section-j} is devoted to studying exponential sums of the form $j_{\lambda, n}$
for general even self-dual lattices of arbitrary signature. 
The results of section \ref{section-j} and related results in appendix \ref{appendix-gauss-sums}
are probably known to experts. But since we could not find a reference,
and we think these results may have applications in other contexts,
we have decided to include the details. In theorem \ref{th-j} we compute
the sums $j_{\lambda,n}$ in terms of Kloosterman sums.
Weil's bound on Kloosterman sums implies that the infinite series for $a_{\lambda}(k,h,s)$
converges for $\op{Re}(s) > 25/2$
and it follows that $E(z,s)$ can be meromorphically continued in this half plane; 
see theorem \ref{th-analytic-continuation-of-E}.
Further, using analytic continuation of Kloosterman sum zeta functions,
we prove that the Dirichlet series $\sum_{n} j_{\lambda,n} n^{-s}$ and
the individual Fourier coefficients $a_{\lambda}(k,h,s)$ have
meromorphic continuation to the whole $s$-plane 
(see \ref{l-analytic-continuation-of-Dirichlet-series-of-j} and
 \ref{remark-analytic-continuation-of-Fourier-coefficients}).
However, since we are unable to find bounds on these functions obtained after
analytic continuation, we cannot prove continuation of $E(z,s)$ beyond
the line $\op{Re}(s) = 25/2$.
\par
We have two motivations for studying $E(z,s)$. 
The first motivation is to find out if the residue of $E(z,s)$ at the ``first pole"
is related to some of the automorphic forms of type $O(25,1)$ constructed by taking Borcherds lift
(see \cite{REB:Aut}, section 10, in particular, the set up of example 10.7).
The second motivation is that the ``complex analytic versions of $E(z,s)$" are
relevant in projective uniformization of a $13$ dimensional complex  ball quotient
whose fundamental group is conjectured to be related to the monster simple group \cite{Allcock-monstrous}.
These complex analytic versions of $E(z,s)$ are instances of the functions considered in \cite{L:CDBA1}
and they are defined as
\begin{equation*}
E^{(l)}(z) = \sum \nolimits_{r \in L(3) } \ip{z}{r}^{-l}
\end{equation*}
where $L$ is the unique
hermitian $\ZZ[e^{2 \pi i/3}]$-lattice of signature $(13,1)$ satisfying $\sqrt{-3} L^{\vee} = L$,
the sum is over $L(3)$ which is the set of roots of $L$, the exponent
$l \in 6 \NN$ and $z$ is a negative norm
vector of $L \otimes_{\ZZ[e^{2 \pi i/3}]} \CC$. 
So $E^{(l)}(z)$ is a section of the tautological line bundle on the complex $13$-ball
$B(L) = \lbrace v \in L \otimes_{\ZZ[e^{2 \pi i/3}]} \CC \colon v^2 < 0 \rbrace/\CC^*$.
Thus $E^{(l)}(z)$ are meromorphic automorphic forms of type $U(13,1)$
invariant under $\op{Aut}(L)$ and having poles along
the mirrors of the complex hyperbolic reflection group $R(L)$. So 
\begin{equation*}
z \mapsto [E^{(30)}(z)/E^{(36)}(z) : E^{(36)}(z)/  E^{(42)}(z) : E^{(42)}(z)/E^{(48)}(z) :  \dotsb]
\end{equation*}
is a meromorphic map from the ball quotient 
$ (B(L) - \lbrace \text{mirrors of $R(L)$} \rbrace) / \op{Aut}(L)$ to the projective space.
Allcock's monstrous proposal states that the orbifold fundamental group of this ball quotient
surjects onto the monster; see  \cite{Allcock-monstrous}, \cite{AB-braidlike} and
the references in there for more details. The calculation of Fourier series of $E^{(l)}(z)$ is
complicated by the fact that the group of translations is a discrete Heisenberg group
rather than an abelian group. We think that the analysis of $E(z,s)$ is a warm-up exercise
for studying the functions $E^{(l)}(z)$.
\par
{\bf Acknowledgement:} I would like to thank Prof. Richard Borcherds
for many stimulating conversations at the beginning of this work.
I would like to thank Prof. Eric Weber and Prof. Manjunath Krishnapur for their help in a 
couple of proofs.
%
%
%
\section{Some exponential sums related to even lattices}
\label{section-j}
\begin{topic}{\bf Notation: }
\label{topic-notation}
We fix some notation that will be used throughout this section.
Let $p$ be a prime. Let $q = p^r$ for some integer $r \geq 1$.
Let $M$ be a nonsingular even lattice of rank $m$.
Except for lemma \ref{l-bound-on-magnitude-of-j},
we do not assume in this section that $M$ is necessarily positive definite.
 Let $n, d \in \ZZ$ and $n \geq 1$.
We abbreviate $M/n = M/n M$.
We define
\begin{equation*}
M_n(d) = \lbrace l \in M/n : l^2/2 \equiv d \bmod n \rbrace.
\end{equation*}
Let $\e(x) = \exp(2 \pi i x)$. Let
$\lambda \in M$. The first goal of this section is to prove theorem \ref{th-j} which lets us
calculate some exponential sums of the form 
\begin{equation*}
j_{\lambda, n}^M(d) = j_{\lambda, n}(d)  = \sum_{l \in M_n(d)} \e (\ip{l}{\lambda}/n)
\end{equation*}
in terms of Kloosterman sums 
\begin{equation*}
S(a,b, n) = \sum_{r \in \ZZ/n : \op{gcd}(r,n)=1 } \e((ar + b \bar{r})/n)
\end{equation*}
and Jordan totient function $J_k(n) = n^k \prod_{p \mid n} (1 - p^{-k} )$. 
As usual,  $\bar{r}$ is the inverse of $r$ modulo $n$, 
that is, $\bar{r}$ denotes any integer such that $r \bar{r} \equiv 1 \bmod n$.
In the definition of $J_k(n)$, the product is over all prime divisors of $n$.
Let $j_{\lambda, n} = j_{\lambda, n}(1)$.
As a consequence of \ref{th-j}, we obtain bounds on $\abs{j_{\lambda, n}}$
in \ref{l-bound-on-magnitude-of-j} and
show in \ref{l-analytic-continuation-of-Dirichlet-series-of-j}
that the Dirichlet series $\sum_{n = 1}^{\infty} j_{\lambda, n} n^{-s}$
can be analytically continued to a meromorphic function on the whole complex
plane.
\par
Given $\lambda \in M$, let $c(\lambda)$ be the largest positive integer such that
$\lambda \in c(\lambda) M$. So $\lambda$ is a primitive vector of $M$
if and only if $c(\lambda) = 1$.
We say $n$ divides $\lambda$ (resp. $n$ and $\lambda$ are relatively prime)
if $n$ divides $c(\lambda)$ (resp. if $n$ and $c(\lambda)$ are relatively prime).
If $p$ is a prime, we let $v_p(\lambda)$ be the $p$-valuation of $c(\lambda)$.
\end{topic}
\begin{theorem} Assume the setup of \ref{topic-notation}.
\par
(a) If $r, n$ are relatively prime positive integers,  then
$M_{n r}(d) \simeq M_r(d) \times M_n(d)$ and 
$j_{\lambda, nr}(d) = j_{\bar{n} \lambda, r}(d)  j_{\bar{r} \lambda, n}(d)$.
\par
(b) If $n \mid \lambda$, then $j_{\lambda, n} (d) = j_{0, n} (d) = \abs{M_n(d)}$.
If $p \nmid d$, 
then  $j_{ p \lambda, p q}(d)  = p^{m-1} j_{\lambda, q}  (d)$.
\par
(c) Assume that $M$ is self-dual and that $n$ and $\lambda$ are relatively prime.
Then
$j_{\lambda, n}(d)  = n^{(m/2) - 1} S(d, \lambda^2/2, n)$. 
\par
(d) Assume $M$ is self-dual. Then
$ j_{0, n}(1) = n^{(m/2) - 1} J_{m/2}(n)$.
\label{th-j}
\end{theorem}
For later application,
we need to calculate $j_{\lambda,n} = j_{\lambda, n}(1)$ for an even self-dual
lattice. Theorem \ref{th-j} is  sufficient for this purpose. 
First, because of part (a), it suffices to calculate
$j_{\lambda, q}$ when $q$ is a prime power.
Next, part (b) lets us reduce to the case when $p \nmid \lambda$ or $\lambda = 0$.
Finally, parts (c) and (d) handle these two cases respectively.
Proofs of parts (c) and (d) require calculation of some quadratic Gauss sums.
These  calculations have been moved back to appendix \ref{appendix-gauss-sums}.
\begin{proof}[proof of \ref{th-j}(a)]
Choose $\bar{n}, \bar{r} \in \ZZ$ such that $\bar{n} n + \bar{r} r = 1$.
From Chinese remainder theorem we have  mutually inverse isomorphisms 
$\pi: M/nr \to M/n \times M/r$ and $\phi: M/n \times M/r \to M/nr$ given by
\begin{equation*}
\pi( l \bmod nr) = (l \bmod n, l \bmod r), \;\;
\phi(l_1 \bmod n , l_2 \bmod r) = ( \bar{n} n l_2 + \bar{r} r l_1) \bmod n r
\end{equation*}
for $l , l_1, l_2 \in M$.
Pick $v = (l_1 \bmod n, l_2 \bmod r) \in M_n(d) \times M_r(d)$.
Then $\phi(v) =  (l \bmod n r)$ where $l =  \bar{n} n l_2 + \bar{r} r l_1$. We have
\begin{equation*}
l^2/2
 \equiv  \bar{n}^2 n^2 (l_2^2/2) +  \bar{r}^2  r^2(l_1^2/2) 
 \equiv  \bar{n}^2 n^2 d + \bar{r}^2 r^2 d 
\equiv d \bmod  n r.
\end{equation*}
So $\phi(v) \in M_{ n r }(d)$. This proves $\phi(M_n (d) \times M_r(d)) \subseteq M_{n r}(d)$.
Clearly $\pi (M_{n r}(d)) \subseteq M_n (d) \times M_r(d)$. So
restrictions of $\pi$ and $\phi$ yield mutually inverse
isomorphisms between $M_{n r}(d)$ and $M_n (d) \times M_r(d)$. 
Now let 
$(l_1,l_2)$
run over $M_n(d) \times M_r(d)$.
Then $\phi(l_1,l_2)$
runs over $M_{n r}(d)$.
It follows that
\begin{equation*}
j_{\lambda, n r}(d) 
= \sum_{l_1 , l_2} \e\bigl(\tfrac{\ip{ \bar{n} n l_2 + \bar{r} r l_1 }{\lambda}}{n r} \bigr)  
= \sum_{l_1 , l_2} \e \bigl(\tfrac{\ip{  l_2}{\bar{n} \lambda}}{ r} \bigr) 
\e\bigl( \tfrac{\ip{l_1}{ \bar{r} \lambda}}{n} \bigr)
= j_{\bar{n} \lambda, r}(d) j_{\bar{r} \lambda, n} (d).
\qedhere
\end{equation*}
\end{proof}
Part (b) of theorem \ref{th-j} can be proved using a Hensel's lemma type argument. We isolate
this argument in the proof of the next lemma.
\begin{lemma} Assume the setup of \ref{topic-notation}.
Further, assume that $M$ is self-dual and $p \nmid d$. Then
the natural projection $\pi: M/p q \to M/q$ maps
$M_{p q } (d)$ onto $M_{q} (d)$.
For each $\bar{l} \in M_q(d)$, there exists $p^{m-1}$ elements 
 $l \in M_{p q}(d)$ such that $\pi(l) = \bar{l}$. 
\label{l-Hensel}
\end{lemma}
\begin{proof} Clearly $\pi(M_{p q}(d) ) \subseteq M_q(d)$.
Given $\bar{u} \in M_{q} (d)$, 
fix $u \in M$ such that  $\bar{u} = u \bmod qM$.
The lifts of $\bar{u}$ to $M/p q$ are of the form
$(u + q x) \bmod p q M$
 where $x$ runs over a full set of coset representatives for $M/p$.
Since $\bar{u} \in M_{q}(d)$, we have $\tfrac{1}{2} u^2 \equiv d \bmod q$. 
Note that $(u + q x) \bmod p q M \in M_{p q}(d)$ if and only if
\begin{equation*}
 \tfrac{1}{2} u^2 + q \ip{u}{x} \equiv d \bmod p q, 
\end{equation*}
or equivalently,
\begin{equation}
\ip{u}{x} \equiv q^{-1} (d - \tfrac{1}{2} u^2) \bmod p.
\label{eq-ipux}
\end{equation}
Since $M$ is self-dual, the bilinear form $\ip{x}{y} \bmod p$ on the $\FF_p$ -vector space
$M/p$ is non-degenerate. 
Note that $(u \bmod p M)$ is a nonzero vector in $M/p$ Since
$\tfrac{1}{2} u^2 \equiv d \bmod q$ and $p \nmid d$.
So the set of all possible $x \in M/p$ satisfying
equation \eqref{eq-ipux} forms an affine hyperplane in the $m$ dimensional 
$\FF_p$ -vector space $M/p$. 
\end{proof}
\begin{proof}[proof of theorem \ref{th-j}(b)]
The first claim is obvious from the definition of $j_{\lambda, n}(d)$. 
The second claim follows from  lemma \ref{l-Hensel}, since
\begin{equation*}
j_{p \lambda, p q}(d) 
= \sum_{ l \in M_{p q}(d)} \e( \ip{l}{\lambda}/ q ) 
= p^{m - 1} \sum_{ \bar{l} \in M_{q} (d)} \e( \ip{\bar{l}}{\lambda}/ q ) 
= p^{m-1} j_{\lambda, q}  (d) .
\end{equation*}
\end{proof}
Given a function $f : \ZZ/n \to \CC$, let 
$(\F f)(x) = \sum_{y \in \ZZ/n} f(y) \e(x y/n)$
denote its Fourier transform.
Lemma \ref{l-Fj} calculates the Fourier transform of 
$d \mapsto j_{\lambda, q}(d)$.
Then we recover $j_{\lambda, q}(d)$ by using the Fourier
inversion formula $(\F^2 f)(x) = n f(-x)$.
\begin{lemma}
Assume the setup of \ref{topic-notation} and that $M$ is self-dual.
 Let $\lambda \in M$.
\par
(a) We have  
\begin{equation}
\F j_{\lambda, q}(c) =
\sum_{l \in M /q }  \e ( (2 \ip{l}{\lambda} + c l^2)/2q ).
\label{eq-Fj}
\end{equation}
\par
(b) Assume $p \nmid \lambda$. If $p \mid c$, then $\F j_{\lambda, q} (c) = 0$. 
If $p \nmid c$, then
\begin{equation}
\F j_{\lambda, q} (c) =  \e( - \bar{c} \lambda^2/2 q ) q^{m/2}.
\label{eq-Fj-rel-prime-case}
\end{equation}
\label{l-Fj}
\end{lemma}
\begin{proof} 
(a) Note that $M/q$ is the disjoint union of $M_q(1), M_q(2), \dotsb, M_q(q)$.
Part (a) now follows since
\begin{equation*}
 \F j_{\lambda, q}(c) 
= \sum_{d = 1}^q \sum_{\substack{l \in M /q \\ l^2/2 \equiv d \bmod q}}
\e\bigl( \tfrac{\ip{l}{\lambda}}{q} \bigr) \e \bigl(\tfrac{c d}{q} \bigr)
= \sum_{l \in M /q } 
\e\bigl( \tfrac{\ip{l}{\lambda}}{q} + \tfrac{c l^2}{2q} \bigr).
\end{equation*}
\par
(b)  If $\op{gcd}(c,q) = 1$, we can complete squares in equation \eqref{eq-Fj} to get
\begin{equation*}
 \F j_{\lambda, q}(c) 
=  \sum_{l \in M/q }  \e\bigl( \tfrac{-\bar{c} \lambda^2}{2q } + \tfrac{c(l + \bar{c}\lambda)^2}{2q} \bigr)
= \e \bigl( \tfrac{-\bar{c} \lambda^2}{2q} \bigr) \sum_{l \in M /q }  \e \bigl( \tfrac{c l^2}{2q} \bigr).
\end{equation*}
Equation \eqref{eq-Fj-rel-prime-case} now follows from the lemma \ref{l-theta-q} in the appendix.
\par
Now assume that $c = p c_1$ for some $c_1 \in \ZZ$.
Let $u$ run over a set of coset representatives for $M/p^{r-1} $
and $x$ run over a set of coset representatives of $M/ p $. Then $(u + p^{r-1} x)$ runs over
over a set of coset representatives of $M/p^r$. 
So, substituting $ l = (u + p^{r-1} x)$  in equation 
\eqref{eq-Fj}, we  
obtain
\begin{align*}
  \F j_{ \lambda, p^r}(c) 
 = \sum_{u \in M/p^{r-1} } 
\e \bigl( \tfrac{\ip{u}{\lambda}}{p^r} + \tfrac{c_1 u^2}{2p^{r-1} } \bigr) \sum_{x \in M/p} 
\e \bigl( \tfrac{\ip{x }{\lambda}}{ p} \bigr). 
\end{align*}
If $p \nmid \lambda$, then
$x \mapsto \ip{x }{\lambda}/ p$ is a nontrivial character on $M/p$, since $M$ is self-dual. 
So the inner sum in the last expression vanishes.
\end{proof}
\begin{proof}[proof of theorem \ref{th-j}(c)]
When $n = q$ is  a prime power, the formula for $j_{\lambda, n}(d)$ follows from 
lemma \ref{eq-Fj-rel-prime-case} using
Fourier inversion. The
general case follows from this using multiplicativity property of $j_{\lambda, n}$ 
and Kloosterman sums (see \ref{th-j}(a) and \cite{IK:ANT} respectively).
\end{proof}
\begin{proof}[proof of theorem \ref{th-j}(d)]
Both sides of the formula we want to prove are multiplicative. So it is enough to 
prove the formula when $n = q = p^r$ is a prime power.
From equation \eqref{eq-Fj} and lemma \ref{l-theta-q},
we have
\begin{equation*}
\F j_{0,p}(c) = \sum_{l \in M/p } \e( cl^2/ 2p) 
= \begin{cases} p^{ m/2} & \text{\; if \;} 1 \leq c < p \\
p^m & \text{\; if \;} c = p.
\end{cases}
\end{equation*}
Using Fourier inversion, we find  
\begin{equation*}
p j_{0,p} (1) 
=  \sum_{c = 1}^p \F j_{0,p}(c) \e(-c/p) 
= (p^m - p^{m/2}). 
\end{equation*}
Lemma \ref{l-Hensel} implies
$j_{0,pq}(1) =  \abs{J_{p q}(1)} = p^{m-1}  \abs{J_{q }(1)} = p^{m-1} j_{0,q}(1)$.
Part (d) of \ref{th-j} follows from this.
\end{proof}
Recall the notation $j_{\lambda,n} = j_{\lambda,n}(1)$.
\begin{lemma} 
Assume that $M$ is self-dual.
Let $u = \min \lbrace v_p(n), v_p(\lambda) \rbrace > 0$.
Then 
\begin{equation*}
 j_{\lambda, n} =  \begin{cases}
p^{(m-1)u} j_{p^{-u} \lambda, p^{-u} n}  &\text{\; if \;}  v_p(\lambda) < v_p(n), \\
 p^{(m-1)u} (1 - p^{-m/2}) j_{p^{-u} \lambda, p^{-u} n}  & \text{\; if \;} v_p(\lambda) \geq v_p(n).
\end{cases}
\end{equation*}
\label{l-cancel-common-factor-in-j}
\end{lemma}
\begin{proof} This is a straight forward computation using \ref{th-j} (a), (b) in the first case and \ref{th-j} (a), (d) in the second case.
\end{proof}
\begin{lemma} Assume $M$ is positive definite and self-dual.
Fix $\epsilon > 0$. Then there exists a constant $C_{\epsilon}$ such 
that
\begin{equation*}
\abs{j_{n, \lambda}} <  C_{\epsilon} \abs{\lambda}^{(m-1)/2 } n^{\epsilon  + (m-1)/2 }
\text{\; for all \;} n \in \NN \text{\; and all \;} \lambda \in M - \lbrace  0 \rbrace.
\end{equation*}
\label{l-bound-on-magnitude-of-j}
\end{lemma}
\begin{proof}
Since $M$ is positive definite and integral, $\lambda^2  \geq \op{gcd}(n, \lambda)^2$. So
it suffices to prove an inequality of the form
\begin{equation}
\abs{j_{n, \lambda}} <  C_{\epsilon} (n \op{gcd}(n, \lambda))^{(m-1)/2} n^{\epsilon}.
\label{eq-bound-on-j}
\end{equation}
We prove \eqref{eq-bound-on-j} by induction on $\op{gcd}(n,\lambda)$.
If $\op{gcd}(n, \lambda) = 1$, then \eqref{eq-bound-on-j} follows 
from \ref{th-j}(c) and Weil's bound for Kloosterman sum (see \cite{IK:ANT}, p. 280)
\begin{equation*}
\abs{S(a,b,n) } \leq \bigl( \sum\nolimits_{d \mid n} 1 \bigr) \op{gcd}(a,b,n)^{1/2} n^{1/2} 
\end{equation*}
and the estimate $\sum_{d \mid n} 1 = o( n^{\epsilon})$ for any $\epsilon > 0$.
\par 
If $\op{gcd}(n, \lambda) > 1$, then choose a prime $p$ dividing $\op{gcd}(n, \lambda)$. 
Lemma \ref{l-cancel-common-factor-in-j}
implies that 
$\abs{j_{\lambda, n} } \leq p^{(m-1)u} \abs{j_{p^{-u} \lambda, p^{-u} n}(d) }$
where $u  = \min \lbrace v_p(n), v_p(\lambda)\rbrace$.
Now one can induct.
\end{proof}
We end this section by showing that the Dirichlet series
\begin{equation*}
L(j_{ \lambda}, s) = \sum_{n \in \NN} j_{\lambda, n} n^{-s}
\end{equation*}
can be meromorphically continued to $\CC$ when the lattice $M$ is self-dual. 
This follows from Selberg's theorem on analytic continuation
of Kloosterman sum zeta function. The specific result we need is
recorded in the next lemma. 
\begin{lemma}
The series $\sum_{n \colon \op{gcd}(n, k) = 1} S(a,b, n) n^{-s}$ can be analytically continued to a meromorphic function on the
whole $s$-plane. It is analytic for $\op{Re}(s) > 1$, except possibly for a finite set of poles on the real segment $1 < s < 2$.
\label{l-Selberg}
\end{lemma}
\begin{proof} 
 In the notation used in
 equation (3.9) of \cite{AS:FCM}, the Dirichlet series
 $F(d) =  \sum_{n \equiv 0 \bmod d} S(a,b, n) n^{-s}$ is 
 equal to $Z(s/2,a,b,  \chi, \Gamma_0(N))$ for the trivial multiplier 
 system $\chi = 1$. By results of \cite{AS:FCM}
 the series $F(d)$ can be analytically continued to a meromorphic function on $\CC$
 (also see \cite{DI:KS} for details).
The lemma follows since
$\sum_{n : \op{gcd}(n,k) = 1 } S(a,b, n) n^{-s} = \sum_{d \mid k} \mu(d) F(d)$.
\end{proof}
\begin{lemma} 
Assume the setup of \ref{topic-notation}. Further, assume that $M$ is self-dual.
Let $T$ be a finite set of primes. Let $\NN_T$ be the set of all positive integers $n$
such that $ p \nmid n$ for all $p \in T$.
Let $\lambda \in M - \lbrace  0 \rbrace$. Then the Dirichlet series
\begin{equation*}
L_T( j_{\lambda}, s) = \sum_{n \in \NN_T } j_{\lambda, n} n^{-s}
\end{equation*}
can be analytically
continued to a meromorphic function on the whole $s$-plane.
For $\lambda  = 0$, we have 
$L(j_{ 0} , s) = \zeta( s + 1 - m)/ \zeta( s + 1 - (m/2))$.
\par
In particular, $L(j_{\lambda}, s)$ has analytic continuation as a meromorphic function on 
 the whole $s$-plane for all $\lambda \in M$.
\label{l-analytic-continuation-of-Dirichlet-series-of-j}
\end{lemma}
\begin{proof}
Let $P(\lambda)$ be the set of prime divisors of $\lambda$. 
We shall prove the lemma by induction on $\abs{P(\lambda) - T }$.
First, assume that $T \supseteq P(\lambda)$. 
If $n \in \NN_T$, then $n$ and $\lambda$ are relatively prime.
So by theorem \ref{th-j}(c), we have
\begin{equation*}
L_T( j_{\lambda}, s) =  \sum_{n \in \NN_T} S(1, \lambda^2/2, n) n^{(m/2) - 1 - s}.
\end{equation*}
Note that $\NN_T = \lbrace n \in \NN \colon \op{gcd}(n, k_T) =1 \rbrace$
where $k_T$ is the product of all the primes in $T$.
So \ref{l-Selberg} proves the lemma in the case $\abs{ P(\lambda) - T } = 0$
\par
Now fix a $\lambda \in M$ and  a finite set of primes $T$ such that
$P(\lambda) - T \neq \emptyset$. We want to prove the lemma for $(\lambda , T)$.
Assume by induction that the lemma holds for all $(\lambda', T')$
such that $\abs{ P(\lambda') - T' } <  \abs{ P(\lambda) - T} $. 
Choose a prime $p \in P(\lambda) - T$. Let $T' = T \cup \lbrace p \rbrace$.
Write $\lambda = p^r \mu$  with $p \nmid \mu$. 
Write
$L_T(j_{ \lambda}, s)  = A_{\infty} + \sum_{t = 0}^{r} A_t$
where
\begin{equation*}
A_t =  \sum_{n \in \NN_T : v_p(n) = t} j_{\lambda, n} n^{-s} 
\text{\; and \;} 
A_{\infty} =  \sum_{n \in \NN_T : v_p(n) > r} j_{\lambda, n} n^{-s}.
\end{equation*}
We shall argue that each $A_j$ can be analytically continued.
Fist consider the sum $A_t$ with $t \leq r$. 
So let $n \in \NN_T$ such that $v_p(n) = t \leq r$. Let $ n_1 = p^{-t} n$.
Using lemma \ref{l-cancel-common-factor-in-j}, we obtain
\begin{equation*}
A_t 
= \sum_{n_1 \in \NN_{ T' } }  j_{0,p^t} \cdot  j_{p^{-t} \lambda, n_1} (p^t n_1)^{-s}
= p^{-t s}  j_{0,p^t}   \cdot L_{T'} (j_{ p^{-t} \lambda} , s).
\end{equation*}
Note that $\abs{ P(p^{-t} \lambda) - T' }= \abs{ (P(\lambda) - T) - \lbrace p \rbrace } 
= \abs{P(\lambda) - T} - 1$. 
So $A_t$ has an analytic continuation by induction hypothesis.
\par
Now consider the sum $A_{\infty}$. Let $n \in \NN_T$ such that $v_p(n)  > r$. 
Lemma \ref{l-cancel-common-factor-in-j} implies
\begin{equation*}
j_{\lambda, n} 
= p^{(m-1) r}  j_{\mu, p^{-r} n}.
\end{equation*}
Write $n_2 = p^{-r} n$. Note that as $n$ varies over
$\NN_T \cap \lbrace k \in \NN \colon v_{p}(k) > r \rbrace$,
the number $n_2$ varies over $\NN_T - \NN_{T'}$.
It follows that
\begin{equation*}
A_{\infty} 
= p^{(m-1) r} \sum_{n_2 \in \NN_{T} - \NN_{T'} } j_{\mu, n_2} (p^{r} n_2)^{-s}
= p^{(m - 1 - s) r} (  L_{T} ( j_{ \mu}, s) - L_{T' } ( j_{ \mu}, s) ).
\end{equation*}
Note that $\abs{P(\mu) - T} = \abs{P(\mu) - T' } = \abs{P(\lambda) - T} - 1$. So
by induction hypothesis, it follows that
$A_{\infty}$ has analytic continuation. This proves the lemma for $\lambda \neq 0$.
\par
Finally, for $\lambda = 0$, theorem \ref{th-j}(d) implies
$L(j_0,s) = L( J_{m/2}, s + 1 - (m/2))$.
Since $\sum_{d \mid n} J_{k} (n) = n^{k}$, we have
$L(J_{k}, w) = \zeta( w - k)/\zeta(w)$. 
\end{proof}
%
%
%
%
\section{The even self-dual lattice of signature (25,1)}
\label{section-II251}
In this section, we collect a few facts about the Leech lattice and $\mathrm{II}_{25,1}$ that
we would need for our calculations in section \ref{section-Fourier-expansion}. 
For more details see \cite{C:Aut26} or \cite{REB:Thesis}.
 Our sign conventions about Lorentzian 
lattices are consistent with \cite{REB:Thesis}.
\begin{topic}{\bf Leech cusp and Weyl chamber ``containing" a Leech cusp:} 
\label{t-C}
As in the introduction, Let $L  = \Lambda \oplus \mathrm{II}_{1,1} \simeq \mathrm{II}_{25,1}$
where $\Lambda$ is the Leech lattice and $\mathrm{II}_{1,1}$ is a hyperbolic cell.
Write $V = L \otimes \RR$. 
If $l \in L$, we shall write $l = (\lambda; m,n)$ with $\lambda \in \Lambda$ and $m,n \in \ZZ$.
Then $l^2 = \lambda^2 - 2 m n$. 
Let $\rho = (0; 0, 1) \in L$. Then $\rho$ is a primitive norm zero vector in $L$ such that 
$\rho^{\bot}/\rho \simeq \Lambda$, that is, $\rho$ is a Leech cusp. 
If $v \in V$, we define its {\it height} by
\begin{equation*}
\op{ht}(v) = - \ip{v}{\rho}.
\end{equation*}
So the height of $( *; h, *)$ is equal to $h$.
The automorphism group of $L$ acts on the hyperbolic space $\cH^{25}$.
As a concrete model of $\cH^{25}$, we take
\begin{equation*}
\cH^{25} = \lbrace v \in V \colon v^2 = -1, \op{ht}(v)  > 0 \rbrace.
\end{equation*}
So if $v \in \cH^{25}$, then it has the form $ v = (x; m, (x^2 + 1)/2m) $ for some $x \in \Lambda \otimes \RR$ and some $m > 0 $.
If $c \in \RR$, define
 \begin{equation*}
 B_{\rho}(c) =  \lbrace v \in \cH^{25} \colon  \op{ht}(v) < c \rbrace.
 \end{equation*}
The sets of the form $B_{\rho}(c)$ are called {\it (open) horoballs around $\rho$}.
Given $x, y \in \cH^{25}$, we say that $x$ is closer to $\rho$ than $y$ if
$\op{ht}(x) < \op{ht}(y)$. This is equivalent to saying that $d(x, B) < d(y,B)$ where
 $B$ is a small horoball around $\rho$ that does not meet $\lbrace x, y \rbrace$.
\par
A {\it root} of a lattice is a norm $2$ lattice vector. Let $r$ be a root of $L$.
We say that $r$ is a {\it positive root} if $\op{ht}(r) > 0$. 
The hyperplane $r^{\bot}$ determines a 
hyperplane in $\cH^{25}$, called the {\it mirror} of the root $r$.
Since the Leech lattice has no roots, 
$\rho$ does not lie on the mirror of any root of $L$.
The positive roots of $L$ corresponding to the mirrors closest to $\rho$ are the roots of height $1$. These
are called the {\it Leech roots} (or {\it simple roots}) and are parametrized by the vectors of the Leech lattice.
For each $\lambda \in \Lambda$, we have a Leech root 
\begin{equation*}
s_{\lambda} =  ( \lambda; 1, \tfrac{\lambda^2}{2} - 1).
\end{equation*}
There is a unique Weyl chamber $C$ of the reflection group $R(L)$ in $\mathcal{H}^{25}$ 
containing $\rho$ in its closure.
Conway proved that
\begin{equation*}
C = \lbrace x \in \cH^{25} \colon \ip{x}{s_{\lambda}} < 0 \text{\; for all \;} \lambda \in \Lambda \rbrace.
\end{equation*}
All these inequalities are necessary to define $C$, that is, 
the walls of $C$ are in bijection with the mirrors
orthogonal to the Leech roots. 
\end{topic}
\begin{lemma}
The horoball $B_{\rho}(1/\sqrt{2})$ is contained in the Weyl chamber $C$.
\label{l-small-horoballs-fit-in-Weyl-chamber}
\end{lemma}
\begin{proof}
Take $v = (\alpha; h, \tfrac{\alpha^2 +1}{2 h} ) \in B_{\rho}(1/\sqrt{2})$. Then $h^2 < \tfrac{1}{2}$.
Writing out $\ip{s_{\lambda}}{v}$ and completing square one obtains
\begin{equation*}
\ip{s_{\lambda}}{v} = -\tfrac{h}{2} ( \lambda - \tfrac{\alpha}{h} )^2 + h - \tfrac{1}{2 h}.
\end{equation*}
Now $h^2< \tfrac{1}{2}$ implies $h - \tfrac{1}{2h} < 0$. So
$\ip{s_{\lambda}}{v} < 0$ for all $\lambda \in \Lambda$.
\end{proof}
\begin{topic}{\bf The translations:}
Let $\TT_{\RR}$ be the automorphisms of $V$ that fix $\rho$ and act trivially on
$\rho^{\bot}/\rho$.
One verifies that $\TT_{\RR} = \lbrace T_v : v \in \Lambda \otimes \RR \rbrace$ where 
\begin{equation*}
T_v(l; a,b) = (l + av; a, b + \ip{l}{v} + av^2/2).
\end{equation*}
The group $\TT_{\RR}$ is naturally isomorphic to the additive group
$\Lambda \otimes \RR$ via $ T_v \mapsto v$.
Let 
\begin{equation*}
\TT = \TT_{\RR} \cap \Aut(L) = \lbrace T_{\mu} : \mu \in \Lambda \rbrace,
\end{equation*}
 that is,
$\TT$ consists of the automorphisms of $L$ that fix $\rho$ and act trivially on $\rho^{\bot}/\rho$. 
\end{topic}
\begin{topic}{\bf The orbits of the continuous group of translations acting on $V$: }
\label{topic-continuous-group-of-translations}
The orbits of action of $\TT_{\RR}$ on $V$ are as follows:
The subspace $\rho^{\bot}$ splits into the zero dimensional orbit $\lbrace \rho \rbrace$
and one dimensional orbits $(l,0,0) + \RR \rho$, for
$l \in \Lambda - \lbrace 0 \rbrace$.
On the complement of $\rho^{\bot}$, the action of $\TT_{\RR}$ is free, so we have
$24$ dimensional orbits. 
For each $h, k \in \RR$, we have a free orbit $S_{k,h}$ defined by
\begin{equation*}
S_{k,h}  = \TT_{\RR}( 0;h,k/2h)= \lbrace v \in V: v^2 = -k , \op{ht}(v) = h \rbrace.
\end{equation*}
The map 
\begin{equation}
\phi: \Lambda \otimes \RR \to S_{k,h} \text{\; defined by \;} \phi(v) = T_v(0;h,k/2h)
\label{eq-phi}
\end{equation}
is an isomorphism with inverse given by $ (l, h, *) \mapsto l/h$. 
Under this identification, the action of $\TT$ on $S_{ k , h}$ corresponds to 
the action of $\Lambda$ on $\Lambda \otimes \RR$ by translation.
So $S_{k,h }/\TT$ can be identified with the torus $(\Lambda \otimes \RR)/  \Lambda$.
\end{topic}
\begin{topic}{\bf The orbits of the discrete group of translations acting on the roots: }
\label{topic-discrete-group-of-translations}
Since $\Lambda$ does not have any roots, all roots of $L$ have nonzero height.
Fix $n \geq 1$.
A root $r \in L(2)$ of height $n$ has the form $r = (l; n, \tfrac{l^2/2 - 1}{n} )$.
So, given $l \in \Lambda$, there is a root of the form $(l; n , *)$
if and only if $l^2/2 \equiv 1 \bmod n$. 
As in \ref{topic-notation}, we write
\begin{equation*}
\Lambda_n(1) = \lbrace \bar{l} \in \Lambda/n : \bar{l}^2/2 \equiv 1 \bmod n \rbrace.
\end{equation*}
For each coset $\bar{l} \in \Lambda_n(1)$,
lift $\bar{l}$ to a lattice vector $l$ and fix once and for all a root
$r_{n,\bar{l}} =(l;n,*)$ (recall: the last coordinate of $r_{n , \bar{l}}$ is determined by
the condition $r_{n,\bar{l}}^2 = 2$).
Now one verifies easily that the set of roots of height $n$
is the disjoint union of the free $\TT$-orbits $\TT r_{n, \bar{l}}$,
parametrized by $\bar{l} \in \Lambda_n(1)$. 
\end{topic}
The lemma below is immediate corollary of \ref{l-small-horoballs-fit-in-Weyl-chamber}.
\begin{corollary}
Let $h, k \in \RR_{>0}$ such that $ 2 h^2/k < 1$. If $z \in S_{k,h}$, then  
$z/\sqrt{k}$  is contained in the Weyl chamber $C$.
\label{l-choice-of-h-k}
\end{corollary}
%
%
%
%
\section{Fourier series expansion of the Poincare series}
\label{section-Fourier-expansion}
\begin{topic}
\label{topic-setup-for-Fourier-series-computation}
We continue with the setup of section \ref{section-II251}. 
Fix a branch of logarithm on positive
half plane so that $(-1)^s = \e(s/2)$.
As stated in the introduction,
the infinite series $E(z,s) = \sum_{ r \in L(2)}  \ip{r}{z}^{-s}$ 
converges for $\op{Re}(s) > 25$ and  defines an analytic function in $s$
invariant under $\Aut(L)$. 
Fix, $k, h > 0$ such that 
\begin{equation*}
\nu = 2h^2/k < 1.
\end{equation*}
Then \ref{l-choice-of-h-k} implies that $S_{k,h}$ does not meet any hyperplane
orthogonal to the roots of $L$.
So the restriction of $E(z,s)$ to
$S_{k,h} \simeq \Lambda \otimes \RR \simeq \RR^{24}$ (see \ref{topic-continuous-group-of-translations})
is well defined and invariant under the group of translation $\TT \simeq \Lambda \simeq \ZZ^{24}$.
Thus $E(z,s)$ defines a function on the torus $S_{k ,h}/\TT$.
We shall now calculate the Fourier series of this function.
We call this the Fourier series of $E(z,s)$ at the Leech cusp $\rho$.
The Fourier series has the form:
\begin{equation}
E( z ,s) = \sum_{\lambda \in \Lambda} a_{\lambda} (k,h,s)
\e ( \ip{ \lambda}{ \phi^{-1}(z)} )     
\label{eq-Fourier-series-of-E}
\end{equation}
where $\phi : \Lambda \otimes \RR \xrightarrow[]{\sim} S_{k,h}$ is as in equation \eqref{eq-phi}.
The Fourier coefficients are given by
\begin{equation*}
a_{\lambda}(k, h ,s) 
= \int_{v \in (\Lambda \otimes \RR)/\Lambda} E(\phi(v),s)   
\e ( -\ip{\lambda}{v}  ) d v.
\end{equation*}
\end{topic}
Theorem \ref{th-fourier-coeffients} gives a formula for the Fourier coefficients. The formula
involves the character sums $j_{n, \lambda}  = j_{n, \lambda}^{\Lambda}(1)$
from section \ref{section-j} and the numbers
\begin{equation*}
c_n = \sqrt{\tfrac{k}{h^2} - \tfrac{2}{n^2} }.
\end{equation*}
Note that our choice $2 h^2/k < 1$ ensures that $c_n^2  > 0$ for all $n \geq 1$.
\begin{theorem}
Write $a^*_{\lambda}(k, h, s) = (h/2 \pi)^s \Gamma(s)  a_{\lambda}(k,h,s)$.
Then one has
\begin{equation*}
 a_{\lambda}^*(k,h,s) 
=
\begin{cases} 2 (1 + \e(-\tfrac{s}{2} ))\abs{\lambda}^{s - 12}
\sum_{n \geq 1}  j_{\lambda, n} n^{-s} c_{n}^{12-s} K_{12 - s}(2 \pi \abs{\lambda} c_{n} ) 
& \text{ if \;} \lambda \neq 0, \\
& \\
(1 + \e(-\tfrac{s}{2} )) \Gamma(s - 12) \pi^{12 - s} \sum_{n \geq 1}  j_{0, n} n^{-s}
c_n^{24 - 2 s}
& \text{ if \;} \lambda = 0.
\end{cases}
\end{equation*}
\label{th-fourier-coeffients}
\end{theorem}
\begin{proof}
Take $z \in S_{k,h}$. Then \ref{l-choice-of-h-k} implies that $z/ \sqrt{k}$ belongs to the Weyl chamber $C$.
So by Conway's theorem (see \cite{REB:L}) $z$ and $\rho$ are on the same side of $r^{\bot}$ for every root $r \in L(2)$.
So if $r$ is a positive root, then $-\ip{r}{z} >  0$.
Using this and the decomposition of $L(2)$ into $\TT$-orbits from  \ref{topic-discrete-group-of-translations}, we obtain
\begin{align*}
E(z,s)
= (1 + (-1)^{-s}) \sum_{r : \op{ht}(r) > 0} (-\ip{r}{z})^{-s}  
= (1 + \e(-\tfrac{s}{2} )) 
\sum_{\substack{ n \geq 1 \\ \bar{l} \in \Lambda_n(1) } }
\sum_{T \in \TT} (-\ip{T r_{n,\bar{l}}}{ z})^{-s}.
\end{align*}
So
\begin{equation*}
a_{\lambda}(k, h ,s) 
=(1 + \e(-\tfrac{s}{2} )) \int_{v \in (\Lambda \otimes \RR)/ \Lambda} 
 \sum_{n , \bar{l} }  \sum_{ \mu \in \Lambda}
(-\ip{r_{n,\bar{l}}}{ T_{\mu} \phi( v ) })^{-s} \e( - \ip{\lambda}{v + \mu} ) d v.  
\end{equation*}
Since $T_{\mu} \phi(v) = \phi(v + \mu)$,
we can change the above expression into an integral over the whole vector space
$\Lambda \otimes \RR$ to get:
\begin{equation*}
a_{\lambda}(k,h,s)= (1 + \e(-\tfrac{s}{2} )) \sum_{n \geq 1}  \sum_{\bar{l} \in \Lambda_n(1)}
\int_{v \in \Lambda \otimes \RR } (- \ip{r_{n,\bar{l}} }{ \phi(v) })^{-s}
\e (-  \ip{\lambda}{v} ) d v. 
\end{equation*}
Recall that $r_{n,\bar{l}}=(l;n,\tfrac{l^2 - 2}{2n})$
and $\phi(v) = (vh ;h, \tfrac{ (vh)^2 + k}{2h}  )$.
Expanding the expression for $-\ip{r_{n,\bar{l}} }{ \phi(v) }$
and completing square we get
\begin{equation*}
-\ip{r_{n,\bar{l}} }{ \phi(v) } = \tfrac{n h}{2}( c_n^2 +(v -\tfrac{ l}{n}  )^2 )
\end{equation*}
where
$c_n^2 = (\tfrac{k}{h^2} - \tfrac{2}{n^2} )$.
We substitute $  u = v - \tfrac{l}{n} $ in the integral to get
\begin{equation*}
 a_{\lambda}(k,h,s)
= (1 + \e(-\tfrac{s}{2} ))\sum_{n \geq 1}  \sum_{\bar{l} \in \Lambda_n(1)}
\int_{\Lambda \otimes \RR} (\tfrac{n h}{2})^{-s} (c_n^2 + u^2)^{-s}
\e (- \ip{\lambda}{ u + \tfrac{l}{n}} ) du 
\end{equation*}
or
\begin{equation}
a_{\lambda} (k, h, s)
=(1 + \e(-\tfrac{s}{2} ))(\tfrac{2}{h})^s  \sum_{n \geq 1}  j_{\lambda, n} n^{-s}
\int_{\Lambda \otimes \RR}  (c_n^2 + u^2)^{-s}
\e(- \ip{\lambda}{u}) du
\label{eq-integral-expression-for-fourier-coeff}
\end{equation}
where
\begin{equation*}
j_{\lambda, n} = j_{\lambda, n}^{\Lambda}
= \sum_{l \in \Lambda_n(1)} \e( -\ip{l}{\lambda}/n)
= \sum_{l \in \Lambda_n(1)} \e(\ip{l}{\lambda}/n).
\end{equation*}
Note that the two sums are equal since $l \mapsto -l$ is an involution of $\Lambda_n(1)$.
\par
Assume $\lambda \neq 0$.
Let $\op{Gram}(\Lambda)$ be a Gram matrix of the  $\Lambda$. So
$u^2 = u' \op{Gram}(\Lambda) u$.
Let $\mu = \op{Gram}(\Lambda)^{1/2} \lambda$.
So $ \mu^2 =  \lambda^2$.
Use the substitution $v = \op{Gram}(\Lambda)^{1/2} u $
in the integral in \eqref{eq-integral-expression-for-fourier-coeff} to find
\begin{equation*}
I =  \int_{\Lambda \otimes \RR}  ( c_n^2 + u^2)^{-s}
\e (- \ip{ u }{\lambda}) du
=   \int_{\RR^{24}}  (c_n^2 + v^2)^{-s}
\e ( -(v \cdot \mu) ) dv.
\end{equation*}
Thus $I$ is the Fourier transform of a radial function.
Let $r = \abs{v}$, $\xi = v/\abs{v}$ and $d v = r^{23} d r d \omega(\xi)$, where
$\omega$ is the standard measure on $S^{23}$. Changing to polar co-ordinates we have
\begin{equation*}
I = \int_{r \in \RR_{+}}  r^{23} (c_n^2 + r^2)^{-s} \int_{\xi \in S^{23}} 
\e( (-r\abs{\mu}) (\xi \cdot \mu/\abs{\mu} )  ) d \omega(\xi) d r   
\end{equation*}
The integral over the sphere gives a Bessel function of the first kind.
Using lemma 9.10.2 of \cite{AAR:SF}, we evaluate the integral over the sphere to get
\begin{equation*}
I = \int_0^{\infty}  (c_n^2 + r^2)^{-s} r^{23} 2\pi J_{11}(2 \pi \abs{\lambda} r )(r\abs{\lambda})^{-11} d r.   
\end{equation*}
Substitute $x = 2 \pi \abs{\lambda} r$  to get
\begin{equation*}
I = (2 \pi \abs{\lambda})^{2 s - 12} \abs{\lambda}^{-12} 
\int_0^{\infty}  x^{12} ((2 \pi \abs{\lambda} c_n)^2 + x^2)^{-s}   J_{11}(x ) d x.
\end{equation*}
The last integral can be expressed in terms of the modified Bessel function $K_{\nu}$. For
 $\re(s) > 23/4$, one has
\begin{equation*}
I = 
(2 \pi \abs{\lambda})^{2 s - 12} \abs{\lambda}^{-12} 
(2 \pi \abs{\lambda} c_n)^{12 - s} K_{12 - s} (2 \pi \abs{\lambda} c_n) / (2^{s-1} \Gamma(s)).
\end{equation*}
(see the table of Hankel transforms in 
\cite{EMOT:TOIT}, vol II. p. 24, 
and substitute $y=1$, $\nu = 11$ and $s = \mu + 1 $ in formula (20)).  
Substituting in equation \eqref{eq-integral-expression-for-fourier-coeff} and
simplifying, we get the formula for the Fourier coefficients 
for $\lambda \neq 0$.
\par
For $\lambda = 0$, from equation \eqref{eq-integral-expression-for-fourier-coeff}
we get
\begin{equation*}
a_0(k, h, s)
=(1 + \e(-\tfrac{s}{2} ))(\tfrac{2}{h})^s  \sum_{n \geq 1}  j_{0, n} n^{-s}
\int_{\Lambda \otimes \RR}  (c_n^2 + u^2)^{-s} du \\
\end{equation*}
By changing to polar coordinates and changing variable $r = c_n x$, we obtain
\begin{equation*}
a_{0}(k,h,s) = (1 + \e(-\tfrac{s}{2} ))(\tfrac{2}{h})^s \op{vol}(S^{23}) \sum_{n \geq 1}  j_{0, n} n^{-s} c_n^{24 - 2s} 
\int_0^{\infty}   x^{23} (1 + x^2)^{-s} dx
\end{equation*}
The last integral is equal to $ \Gamma(12) \Gamma(s -12)/ 2 \Gamma(s)$
for $\op{Re}(s) > 12$.
(see the table of Mellin transforms in 
\cite{EMOT:TOIT} Vol I, p. 311 and substitute $h = 1, \alpha = 1$ $s  = 24$ in formula (30)).
Substituting the value of the integral and $\op{vol}(S^{23}) = 2 \pi^{12}/\Gamma(12)$
and simplifying, we get the formula for $a_0( k,h,s)$.
\end{proof}
Using theorem \ref{th-fourier-coeffients}, we now prove that
$E(z,s)$ can be analytically continued. 
The argument is similar to an argument used to prove analytic continuation
of real analytic Eisenstein series (see \cite{B:A}, p. 68-69).
Recall that we write $\nu = 2 h^2/k < 1$.
Note that $c_n^2 = (2/\nu) ( 1 - \nu/n^2)$.
\begin{theorem}
The series $E(z,s)$ can be analytically continued to a meromorphic function
on the half plane $\op{Re}(s) > 25/2$. The only poles of $E(z,s)$
in this half plane comes from the Fourier coefficient $a_0(k,h,s)$.
\label{th-analytic-continuation-of-E}
\end{theorem}
\begin{proof}
From  \ref{l-analytic-continuation-of-Dirichlet-series-of-j}, 
we know $\sum_n j_{0, n} n^{-s} = \zeta(s - 23)/\zeta(s - 11)$; so this Dirichlet series
has meromorphic continuation to $\CC$. 
Lemma 
\ref{l-analytic-continuation-of-perturbation} 
(applied with $f(n) = j_{0,n}$ and $g(s,y) = (1 - \nu y^2)^{12 - s}$)
implies that
the infinite series for $a_0(k,h,s)$ given in theorem \ref{th-fourier-coeffients} has an 
meromorphic continuation to $\CC$. Now let $\lambda \neq 0$.
The bound on $j_{\lambda,n}$ from lemma \ref{l-bound-on-magnitude-of-j}
shows that the infinite series for $a_{\lambda}(k,h,s)$ in \ref{th-fourier-coeffients}
converges absolutely and uniformly on compacta, for $\op{Re}(s) >25/2$. 
So each Fourier coefficient defines an analytic function on $\op{Re}(s) > 25/2$.
Let  $s $ vary on a compact subset $A$ of $\lbrace z : \op{Re}(z) > 25/2 \rbrace$.
Let $\kappa = \sqrt{2(1  - \nu)/\nu} > 0$. 
Note that $c_n \geq  \kappa$ for all $n$.
So 
\begin{equation*}
\abs{ c_n^{12 - s} } = (1/c_n)^{\op{Re}(s - 12)} \leq (1/\kappa)^{\op{Re}(s - 12)}
\text{\; and \;} 
e^{- \pi \abs{ \lambda} c_n}  \leq e^{- \pi \kappa \abs{\lambda}}.
\end{equation*}
The Macdonald Bessel function $K_s(y)$ decays exponentially as $y \to \infty$.
It is convenient to use the bound $\abs{K_s(y)} \leq e^{-y/2} K_{\op{Re}(s)} (2)$
for $y > 4$ (see \cite{B:A}, p. 66). 
Fix  $\epsilon >0$ such that $\op{Re}(s) > 2 \epsilon + 25/2$ for all $s \in A$.
From \ref{l-bound-on-magnitude-of-j},
we know that there exists a constant $C_{\epsilon}$ such that
$\abs{j_{\lambda,n}} \leq C_{\epsilon} \abs{\lambda}^{23/2} n^{23/2 + \epsilon}$. 
Using the expression for $a_{\lambda}(k,h,s)$ from \ref{th-fourier-coeffients} and 
the bound for $\abs{j_{n,\lambda}}$, we find
 \begin{align*}
 \abs{a_{\lambda}(k,h,s)}
 \leq
 C(s) \abs{\lambda}^{\op{Re}(s) - 1/2} e^{- \pi \kappa \abs{ \lambda}}
 \end{align*}
where 
$C(s)$   is a fixed positive real valued
 continuous function on $ A$ 
 that does not depend on $\lambda$.
 Since $C(s)$ and $\op{Re}(s)$ 
 stay bounded on a compact set,
$\abs{a_{\lambda}(k,h,s)}$ decays exponentially as a function of $\abs{\lambda}$ as
 $\abs{\lambda} \to \infty$. It follows that the Fourier series for $E(z,s)$ given in 
 equation \eqref{eq-Fourier-series-of-E} converges
 absolutely and
 uniformly on compact subsets of $\op{Re}(s) > 25/2$.
 \end{proof}
\begin{remark}
\label{remark-analytic-continuation-of-Fourier-coefficients}
Lemma \ref{l-analytic-continuation-of-perturbation} applied with
$f(n) = j_{n, \lambda}$ and
\begin{equation*}
g(s, y) = \sqrt{1 - \nu y^2}^{12-s} K_{12 - s}( 2 \pi \abs{\lambda} \sqrt{2/\nu} \sqrt{1 - \nu y^2} )
\end{equation*}
implies that each Fourier coefficient $a_{\lambda}(k,h,s)$ has
analytic continuation to a meromorphic function on the whole $s$-plane. But we are not able 
to establish how these extended functions decay as $\abs{\lambda} \to \infty$.
So we are unable to prove analytic continuation of $E(z,s)$ beyond
the line $\op{Re}(s) = 25/2$.
\end{remark}

%
%
%
%
%
\appendix
\section{Some quadratic Gauss sums }
\label{appendix-gauss-sums}
Let $K$ be an even lattice, not necessarily positive definite. 
Let $q = p^r$ where $p$ is a  prime and $r$ is a positive integer.
Let $Q: K/q \to \QQ/\ZZ$ be the quadratic form 
defined by $Q(x) = x^2/2 q$.
Fix an integer $c$ relatively prime to $p$. 
We want to calculate some quadratic Gauss sums of the form
\begin{equation*}
\theta_{q,c}(K) =  \sum_{x \in K/q}  \e(c Q(x)).
\end{equation*}
Note that if $q$ is odd, then $Q$ is a well defined quadratic form on
$K/ q$ for any integral lattice $K$. But if $q = 2^r$,
then $K$ has be even for $Q$ to be well defined.
\par
If $q$ is odd, then the quadratic form $Q$ can be ``diagonalized modulo $q$"
(see lemma \ref{l-odd-diagonalization}).
Then calculation of $\theta_{q,c}$ reduces to calculation of well known one
dimensional Gauss sums.
Proof of lemma \ref{l-odd-diagonalization} is similar to diagonalization
of quadratic forms over $p$-adic integers (see \cite{CS:SPLAG}, Chapter 15, section 4.4).
If $\mathbf{u} = (u_1, \dotsb, u_m)$ is a finite sequence of vectors in $K$, then
$\op{Gram}(\mathbf{u})$ denotes the matrix $(\!( \ip{u_i}{u_j} )\!)$. 
If $\mathbf{u}$ is a $\ZZ$-basis of $K$, then we say $\op{Gram}(\mathbf{u})$ is 
a {\it Gram matrix} of $K$ and the determinant of this matrix is denoted by $\op{det}(K)$.
\begin{lemma}
 Let $q = p^r$ where $p$ is an odd prime and $r$ is a positive integer. 
Let $K$ be an integral lattice of rank $m$ such that $p \nmid \op{det}(K)$.  
Then $K$ has a $\ZZ$-basis $\mathbf{u} = (u_1, \dotsb, u_m)$ such that 
$p \nmid u_i^2$ for all $i$ and $q \mid \ip{u_i}{u_j}$ for all $i \neq j$.
\label{l-odd-diagonalization}
\end{lemma}
\begin{proof}
Claim 1: {\it $K$ has a basis $(u_1, \dotsb, u_m)$ such that
$p \nmid u_j^2$ for some $j$.}
\par
Choose any basis $(u_1, \dotsb, u_m)$  of $K$. 
If possible, suppose $p \mid u_j^2$ for all $j$. Since $p \nmid \op{det}(K)$,
there must exist $i \neq j$, such that $p \nmid \ip{u_i}{u_j}$. 
Then $p \nmid (u_i + u_j)^2$. Replace $u_j$ by $(u_i + u_j)$
in the basis. This proves claim 1.
\par
Claim 2: {\it $K$ has a basis $ (u_1, \dotsb, u_m)$
such that $p \nmid u_1^2$ and $q \mid \ip{u_1}{u_j}$
for $j \geq 2$. } 
\par
Using claim 1, and reindexing  if necessary, we can choose a basis
$(u_1, \dotsb, u_m)$ of $K$ such that $p \nmid u_1^2$. 
Let $k \in \ZZ$ such that  
$k u_1^2 \equiv 1 \bmod q$. Let $u'_j = u_j - k \ip{u_1}{u_j} u_1$ for $j = 2, \dotsb, m$. 
Then $\ip{u_1}{u'_j}  \equiv 0 \bmod q$.
So $(u_1, u'_2, \dotsb, u'_n)$ is a basis of $K$ with the required property. 
This proves claim 2. 
\par
Choose a basis $\mathbf{u} = (u_1, \dotsb, u_m)$ as in claim 2. 
Write $d_1 = u_1^2 \bmod p$.
Let $K_1$ be the $\ZZ$-span of $\lbrace u_2, \dotsb, u_m \rbrace$.
Note that since $q$ divides $\ip{u_1}{u_j}$ for all $j \geq 2$, we have
$\ip{u_1}{K_1} \subseteq q \ZZ$.
Let $G_1= \op{Gram}( u_2, \dotsb, u_m) \bmod p$. 
Note that 
\begin{equation*}
\op{det}(\op{Gram}(\mathbf{u} )) \bmod p = 
\op{det}( \op{Gram} (\mathbf{u}) \bmod p ) 
= \op{det} \bigl( \begin{smallmatrix} d_1 & 0 \\ 0 & G_1 \end{smallmatrix} \bigr)
= d_1 \op{det}(G_1).
\end{equation*} 
Since $\op{det}(\op{Gram}(\mathbf{u} )) \bmod p \neq 0$, 
we must have $\op{det} (G_1) \neq 0$. So $p \nmid \op{det}(K_1)$.
By induction on $m$, $K_1$ has a $\ZZ$-basis satisfying the
conditions of the lemma. Adjoining $u_1$ to this basis yields
 a basis  of $K$ satisfying the conditions of the lemma.
  \end{proof}
  \begin{lemma}
  Assume the setup of lemma \ref{l-odd-diagonalization}. Let $c$ be an integer relatively prime to $p$. Then
  \begin{equation*}
\theta_{q,c}(K)= 
  q^{m/2} \varepsilon_{q}^m  \Bigl( \frac{ (2 \bmod q)^{-1} c}{q} \Bigr)^m \Bigl( \frac{ \op{det}(K)}{q} \Bigr). 
  \end{equation*}
  where 
   $\varepsilon_n = 1$ if $n \equiv 1 \bmod 4$ and $\varepsilon_n = i$ if $n \equiv -1 \bmod 4$.
  \label{l-odd-gauss-sum}
 \end{lemma}
\begin{proof}
Choose a basis $\mathbf{u}$ for $K$ as in lemma \ref{l-odd-diagonalization}. Then
 $q \mid \ip{u_i}{u_j}$ for all $i \neq j$ implies 
\begin{align*}
\theta_{q,c}(K) 
= \sum_{x_1, \dotsb, x_m \in \ZZ/ q} 
\e( c (x_1 u_1 + \dotsb + x_m u_m)^2/ 2 q)  
= \prod_{j} \sum_{x_j \in \ZZ/ q } \e( c u_j^2 x_j^2/ 2 q).
\end{align*}
If $n \in \mathbb{N}$, $\alpha \in \ZZ$ are such that $gcd(2\alpha,n) = 1$, then 
$\sum_{x \in \ZZ/n } \e( \alpha x^2/ n)  = \varepsilon_n \bigl( \frac{\alpha}{n} \bigr)  \sqrt{n}$
(see  \cite{IK:ANT}, page 52). 
Using this and the fact that $p \nmid u_j^2/2$, we find
\begin{align*}
\theta_{q,c}(K)
= \prod_j  \varepsilon_q \Bigl( \frac{c u_j^2/2}{q} \Bigr)  \sqrt{q} 
= q^{m/2} \varepsilon_q^m  \Bigl( \frac{(2 \bmod q)^{-1} c}{q} \Bigr)^m \Bigl( \frac{ u_1^2 \dotsb u_m^2}{q} \Bigr) .
\end{align*}
The  lemma  follows once we note that
$u_1^2 \dotsb u_m^2 \equiv \op{det}(K) \bmod q$.
\end{proof}
  \begin{lemma}
  Let $c$ be an odd integer. 
  Let $K$ be an even self-dual lattice. 
  Then  $\theta_{2^0,c}(K) = 1$ and $\theta_{2^r,c}(K)= 2^{\op{rk}(K)} \theta_{2^{r-2},c}(K)$
  for all $r \geq 2$.
    \label{l-even-gauss-sum}
  \end{lemma}
\begin{proof}
Let $u$ and $y$ run over a set of coset representatives for $K/2^{r-1}$
and $K/2$ respectively. Then $(u + 2^{r-1} y)$ runs over a set of coset
representatives
for $K/ 2^r $. 
Note that 
\begin{align*}
\tfrac{1}{2^{r+1}} c ( u + 2^{r-1} y)^2
\equiv \tfrac{1}{2^{r+1}} c u^2 + \tfrac{1}{2} \ip{u}{y} \bmod \ZZ
\end{align*}
since $c \equiv 1 \bmod 2$, $r - 3 \geq -1$ and $y^2/2 \in \ZZ$  for all $y \in K$.
It follows that
\begin{align*}
\theta_{2^r,c}(K)
= \sum_{u} \e( c u^2/2^{r+1} ) \sum_{y  } \e(   \ip{u}{y}   /2).
\end{align*}
If $y \mapsto  \e( \ip{u}{y}/2)$ is a nontrivial character on $K/2$,
then the inner sum is equal to $0$. 
Note that $y \mapsto  \e( \ip{u}{y}/2)$ is the trivial character if and only if
$\ip{\tfrac{u}{2}}{x} \in \ZZ$ for all $x \in K$.
Since $K$ is self-dual, this is equivalent to saying $\tfrac{u}{2} \in K$.
Letting $u = 2v$, we obtain
\begin{align*}
\theta_{2^r,c}(K)
= \abs{K/2} \sum_{u \in 2 K/ 2^{r-1}} \e \bigl( \tfrac{c u^2}{2^{r+1}}  \bigr)  
=  2^{\op{rk}(K)} \sum_{v \in K/ 2^{r-2}} \e \bigl( \tfrac{c v^2}{ 2^{r-1}} \bigr)  
= 2^{\op{rk}(K)} \theta_{2^{r-2},c}(K).
\end{align*}
\end{proof}
\begin{lemma}
Let $K$ be an even self-dual lattice. Let $p$ be a prime. Let $q = p^r$ for 
some integer $r \geq 1$. Let $c$ be an integer relatively prime to $p$. 
As defined above, let $\theta_{q,c}(K) = \sum_{x \in K/q} \e( c x^2/2 q)$.
One has $ \theta_{q,c}(K) = q^{\op{rk}(K)/2}$.
 \label{l-theta-q}
\end{lemma}
\begin{proof}
First we verify that the lemma holds for $K = \mathrm{II}_{1,1}$ and for $K = E_8$.
\par
First, assume $q$ is odd. When $K = E_8$, the lemma
is immediate from lemma \ref{l-odd-gauss-sum}.
For $K = \mathrm{II}_{1,1}$, lemma \ref{l-odd-gauss-sum} yields
$\theta_{q,c}(\mathrm{I I}_{1,1} ) = q \varepsilon_q^2 \bigl( \frac{-1}{q} \bigr)$.
Note that
both  $\varepsilon_q^2$ and $\bigl( \frac{-1}{q} \bigr)$
are equal to $-1$ if and only if  $q$ is an odd power of a $3$-mod-$4$ prime 
and both are equal to $1$ otherwise. 
It follows that $\theta_{q,c}(\mathrm{I I}_{1,1} ) = q $. 
\par
Now assume $q = 2^r$. We want to verify that 
$\theta_{q,c}( \mathrm{II}_{1,1}) = q$
and $\theta_{2^r,c}( E_8) = q^4$.
In view of lemma \ref{l-even-gauss-sum},
it is enough to check that $\theta_{2,c}(\mathrm{II}_{1,1}) = 2$ and
$\theta_{2,c}(E_8) = 2^4$.
The calculation for $\mathrm{II}_{1,1}$ is trivial.
For the $E_8$ lattice, we know that a convenient complete
set of coset representatives for $E_8/2$ is given
by the zero vector, the vectors of norm $2$ chosen up to sign and one norm $4$
vector chosen from each ``coordinate frame"
of size $16$. In particular 
$\abs{ E_8(0) }+  \tfrac{1}{2} \abs{E_8(2) } +  \tfrac{1}{16} \abs{ E_8(4) }  = 2^{8}$.
It follows that
\begin{align*}
\theta_{2,c}(E_8)
=\abs{E_8(0) } -  \tfrac{1}{2} \abs{ E_8(2) } + \tfrac{1}{16} \abs{ E_8(4) }  
= 2^{8} -  \abs{E_8(2)}  
= 2^{4}.
\end{align*}
\par
Given even lattices $K_1$ and $K_2$, one verifies
$\theta_{q,c}(K_1 \oplus K_2)  = \theta_{q,c}(K_1) \theta_{q,c}(K_2)$. 
In particular, the lemma holds for $K$ if and only if it holds for $K \oplus \mathrm{II}_{1,1}$.
So, adding a hyperbolic cell $\mathrm{II}_{1,1}$ to $K$ if necessary, we may assume without
loss that $K$ is indefinite. The lemma follows by multiplicativity of 
$\theta_{q,c}$, since an indefinite even self-dual lattice
$K$ is isomorphic to a direct sum of certain number of copies of
$E_8$ and $\mathrm{II}_{1,1}$. 
 \end{proof}
%

%
%
%
%
%
\section{A lemma on analytic continuation}
\label{section-analysis}
We prove an elementary lemma showing that under certain conditions,
if a Dirichlet series $L(f,s)$ has analytic continuation, then a certain perturbation
$L^{\sim}_{g}(f,s)$ also has analytic continuation.
\begin{lemma}
Assume $f(n) = O(n^{\alpha - 1})$ for some $\alpha > 1$.
Assume that $L(f,s)= \sum_{n = 1}^{\infty} f(n) n^{-s}$ can be analytically
continued to the whole $s$-plane as a meromorphic function.
Let $g(s,y)$ be a function that is entire in $s$ and 
analytic  for $y$ in an open ball of radius strictly greater than $1$.
Then
\begin{equation*}
L^{\sim}_g(f,s) = \sum_{n = 1}^{\infty} f(n) n^{-s} g(s, 1/n)
\end{equation*}
can be analytically continued to the whole $s$-plane as a meromorphic function.
\label{l-analytic-continuation-of-perturbation}
\end{lemma}
\begin{proof}
Since $f(n)  = O(n^{\alpha - 1})$, 
 $L(f,s)$ converges for $\op{Re}(s) > \alpha$ uniformly on compact sets
and defines a holomorphic function in this domain. Fix $k \in \NN$.
Write the $k$-th order Taylor series expansion of $g$ in powers of $y$ valid for $y$
in a disc of radius greater than $1$:
\begin{align*}
g(s, y) 
=   \sum_{j = 0}^{k-1} \partial_y^{(j)} g(s,0) y^j  + y^kE_k(s,y) 
\text{\; where \;} 
E_k(s,y) =\sum_{j = k}^{\infty}  \partial_y^{(j)} g(s,0) y^{j-k}.
\end{align*}
Here $\partial^{(j)}_y =\tfrac{1}{j!}  \tfrac{\partial^j}{\partial y^j}$.
Substituting the Taylor series of $g$ in $L^{\sim}_g(f,s)$, we get:
\begin{align*}
L^{\sim}_g(f,s) 
=\sum_{j = 0}^{k-1} \partial_y^{(j)} g(s,0) L(f, s + j)   + 
\sum_{n = 1}^{\infty}  a_k(n,s)
\end{align*}
where
$
a_k(n,s) = n^{-s-k} f(n) E_k(s,1/n).
$
The first $k$ terms all have meromorphic continuation to $\CC$ since $L(f,s)$ does.
Note that $E_k(s, y)$ is continuous in $(s,y)$ for $y$
in the closed ball of radius $1$ around $0$ and for $s$ in any half plane.
Let $A$ be a compact subset of
$ \lbrace z \in \CC \colon \op{Re}(z+ k) > \alpha + \delta \rbrace$ for some $\delta > 0$.
There exists constants $C_1, C_2$ such that $\abs{E_k(s,1/n) )} < C_2$ for all 
$s \in A$ and all $n \in \NN$ and such that  $f(n) \leq C_1 n^{\alpha -1}$ for
all $n \in \NN$. It follows that
$\abs{ a_k(n,s) } \leq C_1 C_2/n^{1 + \delta}$.
So
$\sum_{n = 1}^{\infty}  a_k(n,s)$ converges absolutely and uniformly on $A$
and hence it is analytic on $A$.
It follows that $L^{\sim}_g(f,s) $ can be analytically continued to 
the half plane $  \op{Re}(s) > \alpha - k $
as a meromorphic function.
\end{proof}
%
%
%
%
\section{Convergence of the infinite series}
%
%
Let $X$ be a topological space. Let $\Phi$ be a countable set.
For each $r \in \Phi$, let $f_r : X \to \CC$ be a function.
Saying $\sum_{r \in \Phi} f_r(x)$ converges (resp. converges absolutely or
uniformly) means that $\sum_{n  = 1}^{\infty} f_{r(n)}(x)$ converges (resp. converges
absolutely or uniformly) for any bijection $n \mapsto r(n)$ from $\NN$ to $\Phi$,
 and that the limit does not depend on the choice of the bijection.
The goal of this appendix is to write down an elementary proof 
of lemma \ref{l-convergence}. 
\begin{lemma}
Let $L$ be an integral lattice of signature $(l,1)$ and $k > 0$.
Let $A$ be a compact set in $\op{Re}(s) > l$ and $B$ be a compact set in 
$\lbrace v \in L \otimes \RR \colon v^2 < 0 \rbrace$. 
 Then 
the series $E(z,s) = \sum_{r \in L(k)} \langle r, z \rangle^{-s}$ converges absolutely and uniformly 
 for $(z,s) \in B \times A$.
\label{l-convergence}
\end{lemma}
\begin{proof}
Let $V = L \otimes \RR$. We work in the projective model of the hyperbolic space and
write
$\cH^l  = \lbrace v \RR \colon v \in V, v^2 < 0 \rbrace$ where
$v\RR$ is the line containing $v$.
Fix $e_0 \in L$ with $e_0^2 < 0$. 
Then there exists constants $C_0 > 1$ and $\delta > 0$ such that
$\sqrt{-v^2} < C_0$ and $d(v, e_0) < \delta$
for all $v \in B$. Here $d(v, e_0)$ means the hyperbolic distance between 
$v \RR$ and $e_0 \RR$ in $\cH^l$. Let $r \in L(k)$.
Write  $b_r = \abs{\ip{r}{e_0}} (-e_0^2 )^{-1/2} $.
Let $v \in B$. Using triangle inequality in $\cH^l$, we get
\begin{equation*}
\sinh (d(r^{\bot}, v)) \leq  \sinh( d(r^{\bot},e_0) + d(e_0, v))  \leq 
 \sinh\bigl( \sinh^{-1} (k^{-1/2} b_r) + \delta \bigr).
\end{equation*}
Since $\sinh( \sinh^{-1}(a) + \delta) = a \cosh(\delta) + \sqrt{1 + a^2} \sinh(\delta)
\leq a e^\delta + \sinh(\delta) $ for all $a > 0$, we obtain
\begin{equation*}
\abs{ \ip{r}{v} } (- v^2 k )^{-1/2}   \leq k^{-1/2} b_r  e^{\delta} + \sinh(\delta).
\end{equation*}
Since $\sqrt{-v^2} \leq C_0$, we find that
$\abs{ \ip{r}{v} }    \leq  (C_1 b_r + C_2)$ 
for some positive real constants $C_1, C_2$.
Fix $\epsilon > 0$ such that $\op{Re}(s) > l + \epsilon$ for all $s \in A$. Then
\begin{equation*}
\sup \lbrace \abs{ \ip{r}{v}^{-s}} \colon (v,s) \in B \times A \rbrace
 \leq  (C_1 b_r + C_2)^{- l - \epsilon}.
\end{equation*}
Lemmas \ref{l-bound-on-number-of-mirrors-intersecting-a-ball} and \ref{l-summation-by-parts} given below show that 
$\sum_{r \in L(k)}(C_1 b_r + C_2)^{- l - \epsilon}$ converges.
So 
\begin{equation*}
\sum_{r \in L(k)} \sup \lbrace \abs{ \ip{r}{v}^{-s}} \colon (v,s) \in B \times A \rbrace
\end{equation*}
converges. The lemma follows from Weierstrass $M$-test.
\end{proof}
\begin{lemma} Let $L$ be an integral lattice of signature $(l,1)$. 
Let $e_0 \in L$ be a negative norm vector. Let $k > 0$.
For each $r \in L(k)$, define
$b_r = \abs{\ip{r}{e_0}} (-e_0^2)^{-1/2}$.
Let $C_1, C_2$ be constants. Then 
$\abs{ \lbrace r \in L(k) \colon (C_1 b_r + C_2) \leq n \rbrace } = O(n^l)$.
\label{l-bound-on-number-of-mirrors-intersecting-a-ball}
\end{lemma}
\begin{proof}
Let $V = L \otimes \RR$.
Extend $e_0$ to an orthogonal basis $(e_0, \dotsb, e_l)$
of $L \otimes \QQ$. So $e_1, \dotsb, e_l$ have positive norm.
There exists a positive integer $N$ such that
 $L \subseteq N^{-1} (\ZZ e_0 + \ZZ e_1 + \dotsb + \ZZ e_l)$.
Without loss, assume $n > C_2$.
Let $r \in L(k)$ such that $C_1 b_r + C_2 \leq n$.
Write $r = r_0 e_0 + r_1 e_1 + \dotsb + r_l e_l$.  Then 
$r_0^2 (-e_0^2) =   b_r^2$.
Let $C_3 = \min \lbrace e_1^2 , \dotsb, e_l^2 \rbrace > 0$.
Then the norm condition $r^2 = k$ implies
\begin{equation*}
C_3( r_1^2 + \dotsb + r_l^2) \leq \sum_{j = 1}^l r_j^2 e_j^2 = 
 k + r_0^2 (-e_0^2) \leq k + C_1^{-2} ( n - C_2 )^2.
\end{equation*}
It follows that there is a constant $C_4$ such that 
$(N r_1, \dotsb, N r_l)$ is an integer point in an Euclidean ball of radius $C_4n$.
So number of possibilities for $(r_1, \dotsb, r_l)$ is $O(n^l)$.
Once we fix $(r_1, \dotsb, r_l)$,
there are at most two choices for $r_0$, since
$r^2 = k$.
\end{proof}
\begin{lemma}
 Let $\lbrace a_{r} \colon r \in \Phi \rbrace$ be a collection of positive
real numbers indexed by a countable set $\Phi$. 
Let $\alpha_n$ be the cardinality of $\lbrace r \in \Phi \colon a_r \leq n \rbrace$.
Assume that $\alpha_n = O(n^{l})$.
Then the series $ \sum_{r \in \Phi} a_r^{-t}$ converges for all $t > l$.
\label{l-summation-by-parts}
\end{lemma}
\begin{proof} 
Ignoring finitely many terms from the series if necessary,  we may
assume without loss of generality that $a_r > 1$ for all $r \in \Phi$. 
It suffices to show that the partial sums $S_k = \sum_{r : a_r \leq k} a_{r} ^{-t}$
are bounded. Using summation by parts, we find 
\begin{equation*}
S_k
\leq  \sum_{n = 1}^{k-1} (\alpha_{n+1} - \alpha_n ) n^{-t}
= \alpha_k  k^{-t}   + 
 \sum_{n = 1}^{k-1} \alpha_{n+1} (n^{-t} - (n+1)^{-t}).
\end{equation*}
Now use $\alpha_n = O(n^l)$ and
the bound $ (n^{-t} - (n+1)^{-t}) \leq  t n^{-t -1}$. 
\end{proof}
%
%
%
%
%

%
%
\end{document}